\documentclass[11pt,reqno]{amsart}
\usepackage{amssymb, latexsym, mathrsfs, graphicx, fancyhdr,bm}
\usepackage{amsmath}
\usepackage[left=3cm,right=3cm,
    top=2cm,bottom=3cm,bindingoffset=0cm]{geometry}
\usepackage{verbatim}
\usepackage{enumitem}
\usepackage[noadjust]{cite}
\usepackage[table]{xcolor}
\usepackage{multirow}
\usepackage[final]{listings}

\makeatletter
\DeclareOldFontCommand{\rm}{\normalfont\rmfamily}{\mathrm}
\DeclareOldFontCommand{\sf}{\normalfont\sffamily}{\mathsf}
\DeclareOldFontCommand{\tt}{\normalfont\ttfamily}{\mathtt}
\DeclareOldFontCommand{\bf}{\normalfont\bfseries}{\mathbf}
\DeclareOldFontCommand{\it}{\normalfont\itshape}{\mathit}
\DeclareOldFontCommand{\sl}{\normalfont\slshape}{\@nomath\sl}
\DeclareOldFontCommand{\sc}{\normalfont\scshape}{\@nomath\sc}
\makeatother

\usepackage[left=3cm,right=3cm,
    top=2cm,bottom=3cm,bindingoffset=0cm]{geometry}

\usepackage{fancyhdr} % Custom headers and footers
\usepackage{mathrsfs}
\usepackage{cite}

\usepackage{geometry}
\usepackage{tikz}

\usetikzlibrary{calc}
\usetikzlibrary{decorations.pathmorphing}

\numberwithin{equation}{section} % Number equations within sections (i.e. 1.1, 1.2, 2.1, 2.2 instead of 1, 2, 3, 4)
\numberwithin{figure}{section} % Number figures within sections (i.e. 1.1, 1.2, 2.1, 2.2 instead of 1, 2, 3, 4)
\numberwithin{table}{section} % Number tables within sections (i.e. 1.1, 1.2, 2.1, 2.2 instead of 1, 2, 3, 4)

\makeatletter
\def\Ddots{\mathinner{\mkern1mu\raise\p@
\vbox{\kern7\p@\hbox{.}}\mkern2mu
\raise4\p@\hbox{.}\mkern2mu\raise7\p@\hbox{.}\mkern1mu}}
\makeatother

\newcommand{\Sym}{\mathrm{Sym}}

\newcommand{\diag}{\mathrm{diag}}

\newcommand{\Hall}{\mathrm{Hall}}

\title{	Intersection of Solvable Hall subgroups in finite groups
\normalfont \normalsize
}

\author{Anton A. Baykalov, E.P. Vdovin, V.I. Zenkov} % Your name
\date{\today}

\newtheorem{Th}{Theorem}
\newtheorem{Lem}{Lemma}[section]

\theoremstyle{definition}

\newtheorem{Rem}[Lem]{Remark}

\begin{document}
%\refcheck

\maketitle

%----------------------------------------------------------------------------------------
%	PROBLEM 1
%----------------------------------------------------------------------------------------
%\begin{abstract}
%In this paper we obtain a partial solution to Problem 17.41 from "Kourovka notebook": if $S$ is a solvable subgroup of a finite group $G$ with trivial solvable radical, do there exist five conjugate of $ S$ with trivial intersection? We consider the case, when $G$ is a finite almost simple group with simple socle isomorphic to $L_n^{\varepsilon}(q)$ and $S$ is a solvable subgroup of special type.
%\end{abstract}

\section*{Introduction}

Throughout the paper the term ``group'' we always use in the meaning ``finite group''.  We use symbols $A\leq G$ and
$A\unlhd G$ if $A$ is a subgroup of $G$ and  $A$ is a normal subgroup of $G$ respectively.
Given $H\leq G$ by
$H_G=\cap_{g\in G} H^g$ we denote the {\em kernel} of~$H$.

Assume that $G$ acts on $\Omega$. An element $x\in
\Omega$ is called a {\em $G$-regular point}, if $\vert
xG\vert=\vert G\vert$, i.e., if the stabilizer of $x$ is trivial. We define the action of $G$ on $\Omega^k$
by
\begin{equation*}
g:(i_1,\ldots,i_k)\mapsto
(i_1g,\ldots,i_kg).
\end{equation*}
If $G$ acts faithfully and transitively on $\Omega$, then the minimal $k$ such that
$\Omega^k$ possesses a $G$-regular point is called the {\em base size} of $G$ and is denoted by~$\mathrm{Base}(G)$.
For every natural $m$ the number of $G$-regular orbits on  $\Omega^m$ is denoted by
$\mathrm{Reg}(G,m)$ (this number equals $0$ if $m<\mathrm{Base}(G)$). If $H$ is a subgroup of $G$ and $G$
acts on the set $\Omega$ of right cosets of $H$ by right multiplications, then
$G/H_G$ acts faithfully and transitively on $\Omega$. In this case we denote  $\mathrm{Base}(G/H_G)$
and $\mathrm{Reg}(G/H_G,m)$ by $\mathrm{Base}_H(G)$ and $\mathrm{Reg}_H(G,m)$ respectively. We also say that
$\mathrm{Base}_H(G)$ is the {\em base size of}  $G$ {\em with respect to}~$H$. Clearly,
$\mathrm{Base}_H(G)$ is the minimal $k$ such that there exist elements $x_1,\ldots,x_k\in G$ with
$H^{x_1}\cap \ldots\cap H^{x_k}=H_G$. Thus, the base size of $G$ with respect to $H$ is the minimal $k$ such that there exist $k$ conjugates of $H$ with intersection equals~$H_G$.

The following results were obtained in this direction. In 1966 D.S.Passman proved (see \cite{Pass}) that a
$p$-solvable group possesses three Sylow $p$-subgroups whose intersection equals the $p$-radical of $G$. Later in 1996 V.I.Zenkov proved (see \cite{Zen1}) that the same conclusion holds for arbitrary finite group $G$. In \cite{Dolfi} S.Dolfi proved that in every $\pi$-solvable group $G$
there exist three conjugate $\pi$-Hall subgroups whose intersection equals $O_\pi(G)$ (see also~\cite{VdovinIntersSolv}). Notice also that
V.I.Zenkov in \cite{Zen2} constructed an example of a group $G$ possessing a solvable $\pi$-Hall subgroup
$H$ such that the intersection of five conjugates of $H$ equals $O_\pi(G)$, while the intersection of every four conjugates of  $H$
is greater than~$O_\pi(G)$.

In \cite{Zen2} It was conjuctured that if $H$ is a solvable Hall $\pi$-subgroup of a finite group $G$, then $\mathrm{Base}_H(G)\leq 5$. The following theorem allows to reduce the conjecture to the case of almost simple groups.

\begin{Th}{\em \cite[Theorem~1]{VdovinZenkov}}
Let $G$ be a finite group possessing a solvable $\pi$-Hall subgroup $H$. Assume that for every simple component $S$ of $E(\overline{G})$
of the factor group $\overline{G}=G/S(G)$, where $S(G)$ is the solvable radical of $G$, the following condition holds:
\begin{multline*}
\text{for every  } L\text{ such that } S\leq L\leq \mathrm{Aut}(S)\text{ and contains a solvable }\pi-\text{Hall subgroup
}M,\\ \text{ the inequalities }\mathrm{Base}_M(L)\le 5\text{ and }\mathrm{Reg}_M(L,5)\ge5\text{ hold.}
\end{multline*}
Then $\mathrm{Base}_H(G)\le 5$ and $\mathrm{Reg}_H(G,5)\ge5$.
\end{Th}

Later in \cite[Theorem~2]{Vdo2} it was shown that the inequality $\mathrm{Reg}_H(G,5)\geq 5$ holds if $H$ is a solvable Hall $\pi$-subgroup of an almost simple group $G$, whose socle is either alternating, or sporadic, or an exceptional group of Lie type.

We prove the following theorem in the  paper.

\begin{Th}\label{Classical-main}
 Let $S$ be a simple classical group and  $G$ is chosen so  that  $S\leq G\leq \widehat{S}$, where $S$ is a group of inner-diagonal automorphisms of $S$. Assume also that $G$ possesses a solvable Hall subgroup $H$. Then  $\mathrm{Reg}_H(G,5)\le 5$.
\end{Th}

 In view of \cite[Theorem~3]{Vdo2}, if $G$ is a classical group over a field of characteristic $p$ and $H$ is a Hall $\pi$-subgroup of $G$ with $p\in\pi$, then $\mathrm{Reg}_H(G,5)\leq 5$, i.e. Theorem \ref{Classical-main} holds in this case. So we need to prove Theorem \ref{Classical-main} in case $p\not\in\pi$, and we assume that $p\not\in\pi$ below.

\section{ Preliminaries }
Let $\overline{G}$ be a connected reductive algebraic group over algebraically closed field $\overline{\mathbb{F}}_p$ of positive characteristic $p$ and let $\sigma: \overline{G} \to \overline{G}$ be a Frobenius morphism.
%\begin{Def}
If $\overline{H}$ is a $\sigma$-stable subgroup of $\overline{G}$ (so $(\overline{H})^{\sigma}=\overline{H}$), then $\overline{H}_{\sigma}$ denotes the subgroup of $\sigma$-invariant elements of $\overline{H}.$

%Fix a maximal $\sigma$-stable  torus $\overline{T}$ of $\overline{G}$ and denote %by $W$ its Weyl group $\overline{N} / \overline{T}$, where $\overline{N}$ is the %normaliser of $\overline{T}$ in $\overline{G}$. Let $\rho$ be the natural %homomorphism from $\overline{N}$ to $W$. Then $\sigma$ acts on $W$ by $w^{\sigma}%=\rho(n^{\sigma})$ where $w=\To n$, $n \in \No.$ Elements $w_1$ and $w_2 \in W$ %are called { $\sigma$-conjugate} if $w_1=(w^{-1})^{\sigma}w_2w$ for some $w \in %W$.

%\begin{Lem}[{\cite[ Props. 3.3.1 and 3.3.3]{carter0}]}]
%\label{sopr}
%Let $g \in \Go$. A torus $\overline{T}^g$ is $\sigma$-invariant if and only if %$g^{\sigma}g^{-1} \in \overline{N}$. The map $\overline{T}^g \mapsto %\rho(g^{\sigma}g^{-1})$ is a bijection between $\Gs$-conjugacy classes of maximal  %$\sigma$-invariant tori and classes of $\sigma$-conjugate elements of $W.$
%\end{Lem}

%\begin{Lem}[{\cite[Lemma 1.2]{buturl}}]
%\label{butur}
%Let $g^{ \sigma} g^{-1} \in \overline{N} $ and $\rho(g^{\sigma}g^{-1}) = w.$ Then $(\overline{T}^g)_{\sigma} = (\overline{T}_{\sigma w})^g$, where $w$ acts on $\To$ by
%conjugation.
%\end{Lem}

Let $G$ be a finite group such that  $G_0=O^{p'}(\overline{G}_{\sigma})\le G \le \overline{G}_{\sigma}$ (Note that all classical groups can be obtained in this way). Here $O^{p'}(\overline{G}_{\sigma})$ is the subgroup of $\overline{G}_{\sigma}$ generated by all $p$-elements of $\overline{G}_{\sigma}$. Then $T=\overline{T} \cap G$ is a maximal torus of $G$ and $N(G,T)=\overline{N} \cap G$ is the algebraic normaliser of $T$ in $G$.
%
%\end{Def}

In our notation for finite classical groups we follow \cite{kleidlieb}. In particular, $p$ is prime, $q=p^f$ for some positive integer $f$ and ${\bf u}$ is $2$ in unitary case and $1$ otherwise, so the natural module for a classical group is over $\mathbb{F}_{q^{\bf u}}.$ For unification of some formulations we use $GL_n^+(q)$ and $GL_n^-(q)$ for $GL_n(q)$ and $GU_n(q)$ respectively.

% We identify a group of linear transformations of a vector space $V=\mathbb{F}_q^n$ with a group of matrices of transformations in some basis $\beta$. For example, we write $GU_n(q, \beta)$ to specify the basis with respect to which we take matrices of linear transformations from $GU(V)$.

If $n$ is a positive integer, $r$ is an odd prime and $(r,n)=1$, then $e(r,n)$ is minimal positive integer $e$ such that $n^e \equiv 1 \mod r$. If $n$ is an odd integer, then let $e(2,n)=1$ if $n \equiv 1 \mod 4$ and $e(2,n)=2$ if $n \equiv -1 \mod 4.$

\begin{Lem}[{\cite[Lemma 1]{thlsy}}]
\label{inthom}
Let $G$ be a finite group and $A$ its normal subgroup. If $H$ is some Hall $\pi$-subgroup of $G$
then $H \cap A$ is a Hall $\pi$-subgroup of $A$ and $HA/A$ is one in $G/A$.
\end{Lem}

Following P. Hall \cite{thlsy}, we say that a group $G$ is an $E_\pi$-group, if $G$ possesses a Hall $\pi$-subgroup.

\begin{Lem}
\label{sp4}
Let $H \le GSp_4(q)$ such that $H$ stabilises a decomposition
$$V=V_1 \bot V_2$$
with $\dim V_i=2$ and $V_i$ non-degenerate for both $i=1,2.$ Then there exist $x,y,z \in Sp_4(q)$
such that $H \cap H^x \cap H^y \cap H^z \le Z(GSp_4(q)).$
\end{Lem}
\begin{proof}
let ${e_1,f_1, e_2, f_2}$ be a basis of $V$ such that $V_i=\langle e_i, v_i \rangle$ and $(e_i,f_i)=1.$ Let $x$, $y,$ and $z$ be matrices
$$\begin{pmatrix}
1&0&-1&0 \\
0&1&0&0\\
0&0&1&0\\
0&1&0&1
\end{pmatrix}, \text{ }
\begin{pmatrix}
1&0&0&0 \\
0&1&1&0\\
0&0&1&0\\
1&0&0&1
\end{pmatrix} \text{ and }
\begin{pmatrix}
1&0&0&0 \\
0&1&0&1\\
-1&0&1&0\\
0&0&0&1
\end{pmatrix}
$$
respectively in this basis. It is routine to check that $x,y,z \in Sp_4(q).$

Denote $(V_i)x$ by $W_i$ and $(V_i)y$ by $U_i$ for $i=1,2.$ We claim that if $g \in S \cap S^{y} \cap S^{y},$ then $g$ stabilises $V_i$, $i=1,2.$ Assume the opposite, so $(V_1)g=(V_2).$ Therefore, $(W_1)g=W_2$ and $(U_1)g=U_2.$ Thus,
$$(V_1 \cap W_1)g= (V_1)g \cap (W_1)g= (V_2 \cap W_2)$$
and
$$(V_1 \cap U_1)g= (V_1)g \cap (U_1)g= (V_2 \cap U_2).$$
 Notice that $(V_2 \cap W_2)=(V_2 \cap U_2)$ but $(V_1 \cap W_1) \ne (V_1 \cap U_1)$ which is a contradiction since $g$ is invertible. Therefore, $g= \diag[g_1, g_2],$ $g_i \in GL_2(q).$  Also, $g=h^x$ where $h \in S^{x^{-1}} \cap S$, $g=t^y$ where $t \in S^{y^{-1}} \cap S$ and $g=t^z$ where $s \in S^{z^{-1}} \cap S$ It is routine to check that $h=\diag[h_1, h_2]$, $t=\diag[t_1,t_2]$ and $s=\diag[s_1,s_2]$ with $h_i, t_i, s_i \in GL_2(q).$

Now calculations  show that
\begin{align*}
g=
\begin{pmatrix}
h_{(1,1)}      & \multicolumn{1}{c|}{h_{(1,2)}}& 0& 0   \\
0      & \multicolumn{1}{c|}{h_{(1,4)}}& 0 & 0 \\ \cline{1-4}
0      & \multicolumn{1}{c|}{0}& h_{(2,1)} & 0   \\
0      & \multicolumn{1}{c|}{0}& h_{(2,3)} & h_{(2,4)}
\end{pmatrix} & =
\begin{pmatrix}
t_{(1,1)}      & \multicolumn{1}{c|}{0}& 0& 0   \\
t_{(1,3)}      & \multicolumn{1}{c|}{t_{(1,4)}}& 0 & 0 \\ \cline{1-4}
0      & \multicolumn{1}{c|}{0}& t_{(2,1)} & 0   \\
0      & \multicolumn{1}{c|}{0}& t_{(2,3)} & t_{(2,4)}
\end{pmatrix} \\ &=
\begin{pmatrix}
s_{(1,1)}      & \multicolumn{1}{c|}{0}& 0& 0   \\
s_{(1,3)}      & \multicolumn{1}{c|}{s_{(1,4)}}& 0 & 0 \\ \cline{1-4}
0      & \multicolumn{1}{c|}{0}& s_{(2,1)} & s_{(2,2)}   \\
0      & \multicolumn{1}{c|}{0}& 0 & s_{(2,4)}
\end{pmatrix}
\end{align*}
for some $h_{(i,j)}, t_{(i,j)}, t_{(i,j)} \in \mathbb{F}_q$  with
\begin{align*}
h_{(1,1)}=h_{(2,1)};& &t_{(1,1)}=t_{(2,4)}; \\
h_{(1,4)}=h_{(2,4)};& &t_{(1,4)}=t_{(2,1)}.
\end{align*}
So $g$ is scalar and $g \in Z(GSp_4(q)).$
\end{proof}

\section{ Hall subgroups of odd order}
In this section we assume $2,p \notin \pi,$ where $p$ is the characteristic of the base field of a classical group $G$.

\begin{Lem}[{\cite[Theorem A]{gross1}}]
\label{n2conj}
 Suppose the finite group $G$ has a Hall $\pi$-subgroup where $\pi$ is a set of primes not containing $2$. Then all Hall $\pi$-subgroups of $G$ are conjugate.
\end{Lem}

Let $\overline{G}$ be a simple classical algebraic group of adjoint type, $\sigma$ be a Frobenius morphism such that $G_0$ is a finite simple group. Let $G_0 \le G \le \overline{G}_{\sigma},$ so $G$ is an almost simple group. It follows from \cite{gross1} that the group $G$ has a $\pi$-Hall subgroup if, and only if, every composition factor of $G$ has a $\pi$-Hall subgroup. Therefore, we can assume $G= H_1 G_0,$ where $H_1 \in \Hall_{\pi}(\overline{G}_{\sigma}).$ Indeed, if $H \in \Hall_{\pi}(G),$ then there exists $H_1 \in \Hall_{\pi}(\overline{G}_{\sigma})$ such that $H=H_1 \cap G$ by Lemma \ref{inthom} and Theorem \ref{n2conj}. So, if
$$H_1^{g_1} \cap \ldots H_1^{g_k} =1$$
for some $k$ with $g_i \in H_1 G_0$, then $g_i=h_i \cdot s_i$ with $h_i \in H_1$ and $s_i \in G_0.$ Therefore
$$H^{s_1} \cap \ldots H^{s_k} \le H_1^{g_1} \cap \ldots H_1^{g_k} =1.$$
Moreover, by Lemma \ref{inthom} and \cite[Lemma~2.1(e)]{23Hall}, we can assume that
$H$ is a Hall $\pi$-subgroup of $\hat{G} \in \{GL_n(q), GU_n(q), GSp_{2n}(q), GO_n^{\varepsilon}(q)\}.$ and $G= H \cdot (\hat{G} \cap SL_n(q^{\bf u})).$

Criteria for existence and structure of odd order Hall subgroups of classical groups is studied in \cite{gross2}. It is explicetely shown in \cite{oddHall} that, if exists, $\pi$-Hall subgroup of a classical group $G$ of Lie type lies in $N(G,T)$ for some maximal torus $T$.

%In particular they prove the following theorem.
%\begin{Th}[{\cite[Theorem 4]{oddHall}}]
%\label{vdrevth}
%Let $G$ be a finite Lie-type group with a definition field $GF(q)$ of characteristic $p$ which
%is not a Suzuki or Ree group. Let $\pi$ be a set of primes such that $2, p \notin \pi$ and $|\pi \cap \pi(G)| > 2$, let $\rho$ be a set of
%small primes, let $r$ be the smallest prime in $\pi \cap \pi(G)$, and let $\tau = \pi \backslash \{r\}$. Then $G$ possesses property $E_{\pi}$ if, and only
%if, for some cyclotomic polynomial $\Phi_i (t)$ involved in the decomposition of a $p'$-part of the order of $G$, $\pi$
%satisfies one of the following:
%\begin{enumerate}
%\item $\pi \cap \rho = \emptyset$ and $\pi \cap \pi(G) \subseteq \pi(\Phi_i (q))$; here, the Hall $\pi$-subgroup is Abelian and lies in some torus;
%\item $\pi \cap \rho = \{r\}$, $\pi \cap \pi(G) \subseteq \pi(\Phi_i (q))$, and $i = e(r, q)$; here there exists a maximal torus $T$ such that $T$
%contains a Hall $\tau$-subgroup of $G$ and $N (G,T)$ contains a Hall $\pi$-subgroup of $G$;
%\item $\pi \cap  \rho = \emptyset,$ $\tau \cap \pi(G) \subseteq \pi(\Phi_i (q))$, and $|G|_r = |W |_r = |N (G,T )/T |_r$, where $T = Z(C_G (A))$ is a maximal
%torus and $A$ a Hall $\tau $-subgroup of $G$; here, either the Hall $\pi$-subgroup of $G$ is Abelian, or $G$ is not a
%$D \pi$ -group.
%\end{enumerate}
%\end{Th}

\begin{Lem}\label{oddorderHall}
 Let $\hat{G} \in \{GL_n(q), GU_n(q), GSp_{2n}(q), GO_n^{\varepsilon}(q)\}$ with $n \ge 2,3,4,7$ in linear, unitary, symplectic and orthogonal cases respectively. Let $q$ be such that $\hat{G}$ is not solvable.   Let $\pi$ be a set of primes such that $2, p \notin \pi$ and $|\pi \cap \pi(G)| \ge 2$,  let $r$ be the smallest prime in $\pi \cap \pi(G)$, and let $\tau = \pi \backslash \{r\}$. Let $H$ be a Hall $\pi$-subgroup of $\hat{G}$. If $G=H \cdot (\hat{G} \cap SL_n(q^{\bf u})),$ then there exist $x,y,z \in G$ such that $$H \cap H^x \cap H^y \cap H^z \le Z (\hat{G}).$$
\end{Lem}

\begin{proof}
Denote by $r$ the minimal number in $\pi\cap\pi(G)$, and $(\pi\cap \pi(G))\setminus \{r\}$ by $\tau$. Recall, that, by \cite[Theorem 4.9]{gross2}, $\hat{G}$ is a $E_{\pi}$ subgroup if, and only if, $\hat{G}$ is $E_{\{t,s\}}$ for all $t,s \in \pi.$ By Theorem \cite[Theorem 4.6]{gross2}, if $\hat{G} \in \{GL_n(q), GU_n(q), GSp_{2n}(q)\}$, then $H$ has a normal abelian Hall $\tau$-subgroup, $\hat{G}$ satisfies $D_{\tau}$, all $\tau$-subgroups of $\hat{G}$ are abelian  and $e(q,t)=e(q,s)$ for all $t,s \in \tau.$ By \cite[Theorem 4.8]{gross2}, if $\hat{G}= GO_n^{\varepsilon}(q)$, then $H$ has a normal abelian Hall $\tau$-subgroup, $\hat{G}$ satisfies $D_{\tau}$, all $\tau$-subgroups of $\hat{G}$ are abelian  and either $H$ is cyclic or $e(q,t)=e(q,s)$ for all $t,s \in \tau.$

Let $\hat{G}=GL_n(q).$
 By \cite[Theorems 4.2 and 4.6]{gross2}, $\hat{G}$ is a $E_\pi$ group if, and only if, $n<bs$ for every $s\in\tau$, and one of the following is true:
\begin{enumerate}
\item[$(A)$] $a=b$;
\item[$(B)$] $a=r-1,$ $b=r,$ $(q^{r-1}-1)_r=r,$ and $[\frac{n}{r-1}]=[n/r];$
\item[$(C)$]  $a=r-1,$ $b=t,$ $(q^{t-1}-1)_r=r,$  $[\frac{n}{r-1}]=[n/r]+1$, and $n \equiv f-1 \pmod r$;
\item[$(D)$]  $a=r-1,$ $b=1,$ $(q^{r-1}-1)_r=r,$ and $[\frac{n}{r-1}]=[n/r].$
\end{enumerate}

If  $H$ is Abelian, then there exists $x \in G$ such that $$H \cap H^x \le Z(GL_n(q))$$
by \cite[Theorem 1]{zen}. So we assume that $H$ is not abelian, so, by the proof of  \cite[Theorem 4]{oddHall}, a Sylow $r$-subgroup of $\hat{G}$ is not abelian.

Assume that $(A)$ is realised. By the proof of  \cite[Theorem 4]{oddHall}, $H$ lies in the subgroup $G_1=GL_{[n/a]}(q^a)$ of $\hat{G}.$ Precisely, $H$ lies in the group of monomial matrices of $G_1.$ So
$$V=V_1 \oplus \ldots \oplus V_{[n/a]} \oplus W$$
where $\dim V_i=a$ for $i \in \{1, \ldots, [n/a]\}$, $\dim W= n-[n/a] \cdot a$, $W \subseteq C_H(V)$ and $H$ permutes $V_i$.  Therefore $H$ lies in a maximal irreducible group of $H \cdot SL_{[n/a] \cdot a}(q)$ (if $[n/a]>1$, then $H$ lies in an maximal imprimitive subgroup $M \in \mathit{C}_2$; if $[n/a]=1$, then $H$ is abelian) and there exist $x,y,z \in SL_{[n/a] \cdot a}(q) \le SL_n(q)$ such that
$$H \cap H^x \cap H^y \cap H^z \le Z(GL_{[n/a] \cdot a}(q)) \times I_{n-[n/a] \cdot a}$$ by \cite[Theorem 1.1]{burness}.
Notice, that if $a>1$, then $H \cap (Z(GL_{[n/a] \cdot a}(q)) \times I_{n-[n/a] \cdot a})=1;$ if $a=1$, then $[n/a]=n$, so the statement follows in both cases.

Assume that $(B)$ or $(C)$ is realised. By the proof of  \cite[Theorem 4]{oddHall}, $H$ lies in  $$G_1=GL_{[n/r]}(q^r) \times GL_{r-1}(q) \le \hat{G}$$
and
 $$(q^{r-1})_r=|G|_r=|GL_{r-1}(q)|_r=|G_1|_r=r;$$
Also, Hall $\tau$-subgroup of $\hat{G}$ lies in the subgroup of diagonal matrices of $GL_{[n/r]}(q^r).$ Let $V=U \oplus  W$ where $U$ is the natural module for $GL_{[n/r] \cdot r}(q)$ and $W$ is the natural module for $GL_{r-1}(q).$  So $$H=H_{\tau} \times R$$ where $H_{\tau} \le GL_{[n/r] \cdot r}(q)$ stabilises the decomposition
$$U=V_1 \oplus \ldots \oplus V_{[n/r]} \text{ with } \dim V_i=r;$$
and  $R \le GL_{r-1}(q)$ is a cyclic $r$-subgroup. Therefore, as $H$ in the previous case, $H_{\tau}$ lies in the maximal irreducible subgroup of $H_{\tau} \cdot SL_{[n/r] \cdot r}(q)$ and there exist $x_{1},y_1,z_1 \in SL_{[n/r] \cdot r}(q)$ such that
$$H \cap H_{\tau}^{x_1} \cap H_{\tau}^{y_1} \cap H_{\tau}^{z_1} \le Z(GL_{[n/a] \cdot a}(q)) \times I_{n-[n/a] \cdot a}$$ by \cite[Theorem 1.1]{burness}. By \cite[Theorem 1]{zen}, there exist $x_2 \in R \cdot SL_{r-1}(q)$ (so we can assume $x_2 \in SL_{r-1}(q)$) such that $R \cap R^{x_2}=1,$ since $a=r-1>1$, so $R \cap Z(GL_{r-1}(q))=1.$ Let $x=\diag[x_1, x_2]$, $y=\diag[y_1, I_{r-1}]$, $z=\diag[z_1, I_{r-1}]$. It is easy to see that
$$H \cap H^x \cap H^y \cap H^z=1.$$

Assume that $(D)$ is realised.  By the proof of  \cite[Theorem 4]{oddHall}, $H$ lies
in the group of monomial matrices of $\hat{G},$ so $H$ lies in the maximal imprimitive group of $H \cdot SL_{n}(q)$ and there exist $x,y,z \in SL_{n}(q)$ such that
$$H \cap H^x \cap H^y \cap H^z \le Z(GL_{n})$$ by \cite[Theorem 1.1]{burness}.

Let $\hat{G}=GU_n(q).$
 By \cite[Theorems 4.3 and 4.6]{gross2}, $\hat{G}$ is a $E_\pi$ group if, and only if, $n<bs$ for all $s\in\tau$, and one of the following is true:
\begin{enumerate}
\item[$(A)$] $a=b \equiv 0 \mod 4$;
\item[$(B)$] $a=b \equiv 2 \mod 4$ and $2n<bs$ for all $s\in\tau$;
\item[$(C)$] $a=b \equiv 1 \mod 2$;
\item[$(D)$] $r \equiv 1 \mod 4$, $a=r-1,$ $b=2r,$ $(q^{n}-1)_r=r,$ and $[\frac{n}{r-1}]=[n/r];$
\item[$(E)$] $r \equiv 3 \mod 4$, $a=\frac{r-1}{2},$ $b=2r,$ $(q^{n}-1)_r=r,$ and $[\frac{n}{r-1}]=[n/r];$
\item[$(F)$] $r \equiv 1 \mod 4$, $a=r-1,$ $b=2r,$ $(q^{n}-1)_r=r,$ and $[\frac{n}{r-1}]=[n/r]+1$ and
$n \equiv r-1 \pmod r;$
\item[$(G)$] $r \equiv 3 \mod 4$, $a=\frac{r-1}{2},$ $b=2r,$ $(q^{n}-1)_r=r,$ and $[\frac{n}{r-1}]=[n/r]+1$ and
$n \equiv r-1 \pmod r;$

\item[$(H)$] $r \equiv 1 \mod 4$,  $a=r-1,$ $b=2,$ $(q^{n}-1)_r=r,$  $n<2s$ and $[\frac{n}{r-1}]=[n/r]$;
\item[$(I)$] $r \equiv 3 \mod 4$, $a=\frac{r-1}{2},$ $b=2,$ $(q^{n}-1)_r=r,$ $n<2s$ and $[\frac{n}{r-1}]=[n/r].$
\end{enumerate}

If $H$ is abelian, then there exists $x \in G$ such that $$H \cap H^x \le Z(GU_n(q))$$
by \cite[Theorem 1]{zen}. So let $H$ be non-abelian.

In cases $(A)$--$(C)$, by the proof of  \cite[Theorem 4]{oddHall}, $H$ lies in subgroup $G_1=GL_{[n/a]}(q^a)$ of $\hat{G}$ and the statement follows as in case $(A)$ for $\hat{G}=GL_n(q).$

In cases $(D)$--$(G)$,  by the proof of \cite[Theorem 4]{oddHall}, $H$ is abelian.

In cases $(H)$ and $(I),$  by the proof of \cite[Theorem 4]{oddHall}, $H$ lies in the group of monomial matrices of $\hat{G}$ so $H$ lies in the maximal imprimitive group of $G$ and there exist $x,y,z \in SU_{n}(q)$ such that
$$H \cap H^x \cap H^y \cap H^z \le Z(GU_{n}(q))$$ by \cite[Theorem 1.1]{burness}.

Let $\hat{G}=GO_n^{\varepsilon}$.
 By \cite[Theorems 4.4 and 4.6]{gross2}, $\hat{G}$ is a $E_\pi$ group if, and only if, $n<bs$ for all $s\in\tau$, and one of the following is true:
\begin{enumerate}
\item[$(A)$] $\varepsilon=+$, $a=b \equiv 0 \mod 2$ and $n<bs$;
\item[$(B)$] $\varepsilon=+$, $a=b \equiv 1 \mod 2$ and $n<2bs$;
\item[$(C)$] $\varepsilon=-$, $a=b \equiv 0 \mod 2$ and $n<bs$;
\item[$(D)$] $\varepsilon=-$, $a=b \equiv 1 \mod 2$ and $n<bs$;
\item[$(E)$] $\varepsilon=-$, $a \equiv 1 \mod 2$, $b=2a$ and $n=4a$;
\item[$(F)$] $\varepsilon=-$, $b \equiv 1 \mod 2$, $a=2b$ and $n=4b$;
\end{enumerate}

The proof in cases $(A)$--$(D)$ is analogous to the proof for $\hat{G}=GL_n(q)$ in case $(A)$ and for $\hat{G}=GU_n(q)$ in cases $(A)$--$(C).$ In cases $(E)$ and $(D)$, by the proof of  \cite[Theorem 4]{oddHall}, $H$ is abelian.

Let $\hat{G}=GSp_{2n}(q).$ By \cite[Theorem 4.5]{gross2}, $\hat{G}$ is a $E_\pi$ group if, and only if, one of the following is true:
\begin{enumerate}
\item[$(A)$]  $a=b \equiv 0 \mod 2$ and $2n<bs$ for all $s\in\tau$;
\item[$(B)$]  $a=b \equiv 1 \mod 2$ and $n<bs$ for all $s\in\tau$;
\end{enumerate}

In both cases the proof is analogous  to the proof for $\hat{G}=GL_n(q)$ in case $(A)$ and for $\hat{G}=GU_n(q)$ in cases $(A)$--$(C)$ unless $G \le GSp_4(q)$ and $a=2,$ so  $H$  lies in maximal subgroup $M$ stabilising a decomposition of $V$ into two non-degenerate subspaces. In this case $M$ can be a {\it standard} subgroup in terms of \cite{burness}. If it is the case, then the statement  follows by Lemma \ref{sp4}.
\end{proof}

\section{ Hall subgroups of even order}
In this section we assume $2 \in \pi$ and $p\not\in\pi$, where $p$ is the characteristic of the base field of a classical group $G$.

Let $\overline{G}$ be a simple classical algebraic group of adjoint type, $\sigma$ be a Frobenius morphism such that $G_0$ is a finite simple group. Let $G_0 \le G \le \overline{G}_{\sigma},$ so $G$ is an almost simple group.

Assume that $3 \notin \pi.$ It follows from \cite[Conjectures 1.2 and 1.3]{oddHall} (this Conjectures follows from the results of \cite{oddHall})  that if $G$ has a Hall $\pi$-subgroup $H$, then $H$ is solvable and all such subgroups are conjugate in $G$. Also, a finite group $R$ has a $\pi$-Hall subgroup if, and only if, every composition factor of $R$ has a $\pi$-Hall subgroup.

 Therefore, we can assume $G= H \cdot G_0,$ where $H \in \Hall_{\pi}(\overline{G}_{\sigma})$ as in previous section.
Moreover,  by Lemma \ref{inthom} and \cite[Lemma~2.1(e)]{23Hall}, we can assume that
$H$ is a Hall $\pi$-subgroup of $\hat{G} \in \{GL_n(q), GU_n(q), GSp_{2n}(q), GO_n^{\varepsilon}(q)\}.$ and $G= H \cdot (\hat{G} \cap SL_n(q^{\bf u})).$

\begin{Lem}
\label{3notinpi}
Let $3,p \notin \pi$ and $2 \in \pi.$ Let $H$ be a solvable Hall $\pi$-subgroup of $$\hat{G} \in \{GL_n(q), GU_n(q), GSp_n(q), GO_n^{\varepsilon}(q)\}$$ with $n \ge 2,3,4,7$ in linear, unitary, symplectic and orthogonal cases respectively. Let $q$ be such that $\hat{G}$ is not solvable. Let $G_0=SL_n(q^{\bf u}) \cap \hat{G}.$ If $G= H \cdot G_0,$ then there exist $x,y,z \in G$ such that $$H \cap H^x \cap H^y \cap H^z\le Z(\hat{G}).$$
\end{Lem}
\begin{proof}
Let  $H_0=H \cap G_0.$ By \cite[Theorem 5.2]{evenHall}, $H_0$ lies on $N(G_0,T_0)$ where $T_0$ is a maximal torus of $G_0$ such that $N(G_0,T_0)$ contains a Sylow $2$-subgroup of $G_0$ (all such tori are conjugate in $G_0$ by \cite[Lemma 3.10]{evenHall}) and one of the following is realised
\begin{itemize}
\item $e(2,q)=1$ and $\pi \cap \pi (G_0) \subseteq \pi(q-1);$
\item $e(2,q)=2$ and $\pi \cap \pi (G_0) \subseteq \pi(q+1).$
\end{itemize}

It is easy to see that, if $T \ge T_0$ is a maximal torus of $\hat{G}$ containing a Sylow $2$-subgroup, then $H \le N(\hat{G},T),$ since $|N(\hat{G},T)|_{\pi}=|\hat{G}|_{\pi}.$ By \cite[Theorem 1]{sylow2} (or the proof of \cite[Lemma 3.10]{evenHall}), $N(G,T)$, and hence $H$, stabilises a decomposition
\begin{equation}
\label{decompW}
V=V_1 \bot \ldots \bot V_{[k]} \bot W
\end{equation}
where $\dim V_i=2$ and $\dim W \in \{0,1,2\}.$ By that we mean that $H$ stabilises $W$ and permutes $V_i.$ If $\hat{G}$ is unitary, symplectic or orthogonal, then $V_i$-s are pairwise isometric non-degenerate subspace and $W$ is a non-degenerate subspace. In particular, if $\hat{G}$ is orthogonal and $\dim W=2,$ then we assume that $W$ is not of the same type as $V_i$ since otherwise we can take $V_{k+1}:=W.$

If $n=2,$ so $\hat{G}=GL_2(q),$ then $H$ lies in a maximal $\mathit{C}_3$-subgroup $M$ of $G$ and the statement follows by \cite[Theorem 1.1]{burness}.

\medskip

Assume $n>2$ and $\hat{G}$ is not orthogonal. If $n$ is even, then $H$ lies in a maximal imprimitive (stabilising the decomposition \eqref{decompW}) subgroup $M$ of $G,$ so the statement follows by \cite[Theorem 1.1]{burness}  unless $G \le GSp_4(q)$ and the statement  follows by Lemma \ref{sp4}.

Let $n \ge 3$ is odd, so $\hat{G}$ is $GL_n(q)$ or $GU_n(q).$
%If $e(2,q)=1,$ then $N(\hat{G},T)$ consists of monomial matrices by \cite[Theorem 5.2]{evenHall}, so $H$ lies in a maximal imprimitive subgroup  of $G$ and the statement follows by   \cite[Theorem 1.1]{burness}.
%Let $e(2,q)=2.$
Let $\{v_1, \ldots, v_n\}$ be a basis (orthonormal if $\hat{G}=GU_n(q)$) such that $V_i=\langle v_{2i-1}, v_{2i} \rangle$ for $i \in \{1, \ldots, [n/2]\}$ and $W= \langle v_n \rangle.$
Let $\sigma \in \Sym(n)$ be $(1, 2, \ldots, n)$ and
$$x= \mathrm{PermMat}(\sigma) \cdot \diag({\mathrm{sgn}(\sigma), 1, \ldots, 1}) \in SL_n^{\varepsilon}(q).$$ Therefore, $H \cap H^x$ stabilises decompositions \eqref{decompW} and
$$\langle v_2, v_3\rangle \bot \langle v_4, v_5\rangle \bot \ldots \bot \langle v_{n-1}, v_{n}\rangle \bot \langle v_1 \rangle.$$

It is easy to see that $H \cap H^x$ consists of diagonal matrices, so $H \cap H^x$ is abelian. Therefore, by \cite[Theorem 1]{zen}, there exists $y \in G$ such that
$$(H \cap H^x) \cap (H \cap H^x)^y \le Z(G).$$

\medskip

Assume now that ${\hat{G}}$ is orthogonal, so $n \ge 7.$ If $\dim W=0,$ then $H$ lies in a maximal imprimitive (stabilising the decomposition \eqref{decompW}) subgroup $M$ of $G,$ so the statement follows by \cite[Theorem 1.1]{burness}.

Let $\dim W=1,$ so $n$ is odd and $\hat{G}=GO_n(q).$ Let $Q$ be the quadratic form associated with $\hat{G}$ and let $Q(v_n)=\lambda \in \mathbb{F}_q^*$ where $\langle v_n\rangle=W.$ Since $q$ is odd, $Q:V_i \to \mathbb{F}_q$ is surjective (see \cite[\S 2.5]{kleidlieb}), we can choose a basis $\beta_i=\{v_{2i-1}, v_{2i}\}$ of $V_i$ such that $Q(v_{2i-1})= \lambda$ and ${\bf f}(v_{2i-1},v_{2i})=0$ where ${\bf f}$ is the bilinear form associated with $Q.$ Let $\sigma \in \Sym(n)$ be $(1, 3, 5, \ldots,n-2, n)$ and
$$x= \mathrm{PermMat}(\sigma) \cdot \diag({\mathrm{sgn}(\sigma), 1, \ldots, 1}) \in SO_n(q).$$
Therefore, $H \cap H^x$ stabilises decompositions \eqref{decompW} and
$$\langle v_3, v_2\rangle \bot \langle v_5, v_4\rangle \bot \ldots \bot \langle v_{n}, v_{n-1}\rangle \bot \langle v_1 \rangle.$$

It is easy to see that $H \cap H^x$ consists of diagonal matrices, so $H \cap H^x$ is abelian. Therefore, by \cite[Theorem 1]{zen}, there exists $y \in G$ such that
$$(H \cap H^x) \cap (H \cap H^x)^y \le Z(G).$$

Let $\dim W=2$, so $n$ is even and $\hat{G}=GO_n^{\varepsilon}(q).$ By \cite[Lemma 2.5.12]{kleidlieb},
we can choose a basis $\beta_i=\{v_{2i-1}, v_{2i}\}$ of $V_i$ and a basis $\{v_{n-1},v_n\}$ of $W$ such that $Q(v_{2i-1})= 1$ and ${\bf f}(v_{2i-1},v_{2i})=0.$
Let $\sigma \in \Sym(n)$ be $(1, 3, 5, \ldots,n-1)$ and
$$x= \mathrm{PermMat}(\sigma) \cdot \diag({\mathrm{sgn}(\sigma), 1, \ldots, 1}) \in SO_n^{\varepsilon}(q).$$
Therefore, $H \cap H^x$ stabilises decompositions \eqref{decompW} and
$$\langle v_3, v_2\rangle \bot \langle v_5, v_4\rangle \bot \ldots \bot \langle v_{n-1}, v_{n-2}\rangle \bot \langle v_1, v_n \rangle.$$
It is easy to see that $H \cap H^x$ consists of diagonal matrices, so $H \cap H^x$ is abelian. Therefore, by \cite[Theorem 1]{zen}, there exists $y \in G$ such that
$$(H \cap H^x) \cap (H \cap H^x)^y \le Z(G).$$
\end{proof}

\begin{Rem}
Let  $\hat{G}=GL_n(q)$ and let  $H$ be as in Lemma \ref{3notinpi}.  If $n$ even, then, by \cite{james}, there almost always exists just two conjugates of $H$ whose intersection lies in $Z(\hat{G}).$ If $n \ge 5$ is odd, then one can show that $H \cap H^x \le Z(\hat{G})$ where
$$x=
\begin{pmatrix}
1&      1&       0&\ldots& 0\\
0&      1&     1  & 0     & \ldots  \\
&      & \ddots &\ddots& \\
0& \ldots&0       & 1    &1\\
1& 0     & \ldots & 0    &1
\end{pmatrix}.
$$
Using a similar technique, Baykalov in \cite{baykalov} show that, if $R$ is a solvable imprimitive subgroup in $\hat{G}=GU_n(q)$ ($GSp_n(q)$ respectively), then in almost all cases there exist $x$ and $y$ in $SU_n(q)$ ($Sp_n(q)$ respectively) such that $S \cap S^x \cap S^y \le Z(\hat{G})$.
\end{Rem}

\begin{Lem} \label{23inpi}
Let  $p \notin \pi$ and $2,3 \in \pi.$ Let $H$ be a solvable Hall $\pi$-subgroup of $$\hat{G} \in \{GL_n(q), GU_n(q), GSp_n(q), GO_n^{\varepsilon}(q)\}$$ with $n \ge 2,3,4,7$ in linear, unitary, symplectic and orthogonal cases respectively. Let $q$ be such that $\hat{G}$ is not solvable. Let $G_0=SL_n(q^{\bf u}) \cap \hat{G}.$ If $G= H \cdot G_0,$ then there exist $x,y,z \in G$ such that $$H \cap H^x \cap H^y \cap H^z\le Z(\hat{G}).$$
\end{Lem}
\begin{proof}
Assume that $\hat{G}$ is not orthogonal. By \cite[Lemma 4.1]{23Hall}, $H$ stabilises a decomposition
$$V=V_1 \bot \ldots \bot V_k.$$
into a direct sum of pairwise orthogonal non-degenerate (arbitrary if $V$ is linear) subspaces $V_i$ where $\dim(V_i) \le 2$ for $i \in \{1,\ldots,k\}.$ If $\hat{G}=GL_n^{\varepsilon}(q)$, then, by the proof of \cite[Lemma 4.3]{23Hall}, we can assume that either $\dim V_i=1$ for all $i$ or $\dim V_i =2$ for $i<k$ and $\dim V_k \in \{1,2\}$.  If $\hat{G}=GSp_n(q)$, then $\dim V_i=2$ for all $i$ since all one-dimensional subspaces are singular in this case. The rest of the proof is as in Lemma \ref{3notinpi}.

Assume now $\hat{G}=GO_n^{\varepsilon}(q).$ Since $H$ is solvable, one of $(a)$--$(e)$ holds in \cite[Lemma 6.7]{23Hall}. In cases $(a)$--$(c)$, $H$ stabilises a decomposition of  $V$ as in Lemma \ref{3notinpi} and the proof as in Lemma \ref{3notinpi} works. In cases $(d)$ and $(e)$ we have $n=11$ and $n=12$, $H$ stabilises decompositions
$$V=(V_1 \bot V_2 \bot V_3 \bot V_4) \bot (W_1 \bot W_2 \bot W_3) $$
and
$$V=(V_1 \bot V_2 \bot V_3 \bot V_4) \bot (W_1 \bot W_2 \bot W_3) \bot W_4 $$
respectively. By that we mean that $H$ permutes $V_i$-s and $W_i$-s between and stabilises $\sum_{i=1}^4 V_i$, $\sum_{i=1}^3$ and $W_4$. Here $V_i,$ $W_i$ are non-degenerate, $\dim V_i=2$ and $\dim W_i=1.$ As in Lemma \ref{3notinpi}, we can choose the basis $\{v_1, \ldots, v_n\}$ of $V$ such that
$V_i=\langle v_{2i-1}, v_{2i} \rangle$, $W_i= \langle v_{8+i} \rangle$,
$$Q(v_1)=Q(v_3)=Q(v_5)=Q(v_7)=Q(v_9)=Q(v_{10})=Q(v_{11})=Q(v_{12})$$
and ${\bf f}(v_i,v_j)=0$ for $i \ne j$. Let Let $\sigma \in \Sym(n)$ be $(1, 3, 5,9)( 7,10)$ and
$$x= \mathrm{PermMat}(\sigma) \in SO_n^{\varepsilon}(q).$$
Therefore, $H \cap H^x$ stabilises the decomposition above  and
$$(\langle v_3, v_2\rangle \bot \langle v_5, v_4\rangle \bot \langle v_9, v_6\rangle \bot \langle v_{10}, v_8\rangle) \bot (\langle v_1\rangle \bot \langle v_7\rangle \langle v_{11}\rangle) \bot \langle v_{12}\rangle.$$
It is easy to see that $H \cap H^x$ consists of diagonal matrices, so $H \cap H^x$ is abelian. Therefore, by \cite[Theorem 1]{zen}, there exists $y \in G$ such that
$$(H \cap H^x) \cap (H \cap H^x)^y \le Z(G).$$

\end{proof}

Now Theorem \ref{Classical-main} follows by Lemmas \ref{oddorderHall}, \ref{3notinpi}, and \ref{23inpi}.

\newpage

\bibliographystyle{plain}
\bibliography{ref}

\begin{thebibliography}{10}

\bibitem{baykalov}
Anton~A. Baykalov.
\newblock {\em Intersection of conjugate solvable subgroups in classical groups
  of Lie type}.
\newblock PhD thesis, The University of Auckland, 2021 (to appear).

\bibitem{burness}
Timothy~C. Burness.
\newblock On base sizes for actions of finite classical groups.
\newblock {\em J. Lond. Math. Soc. (2)}, 75(3):545--562, 2007.

\bibitem{sylow2}
Roger Carter and Paul Fong.
\newblock The {S}ylow {$2$}-subgroups of the finite classical groups.
\newblock {\em J. Algebra}, 1:139--151, 1964.

\bibitem{Dolfi}
S.~Dolfi.
\newblock Large orbits in coprime actions of solvable groups.
\newblock {\em Trans. AMS}, 360(1):135--152, 2008.

\bibitem{gross1}
Fletcher Gross.
\newblock Conjugacy of odd order {H}all subgroups.
\newblock {\em Bull. London Math. Soc.}, 19(4):311--319, 1987.

\bibitem{gross2}
Fletcher Gross.
\newblock Odd order {H}all subgroups of the classical linear groups.
\newblock {\em Math. Z.}, 220(3):317--336, 1995.

\bibitem{thlsy}
P.~Hall.
\newblock Theorems like {S}ylow's.
\newblock {\em Proc. London Math. Soc. (3)}, 6:286--304, 1956.

\bibitem{james}
J.~P. James.
\newblock Two point stabilisers of partition actions of linear groups.
\newblock {\em J. Algebra}, 297(2):453--469, 2006.

\bibitem{kleidlieb}
Peter Kleidman and Martin Liebeck.
\newblock {\em The subgroup structure of the finite classical groups}, volume
  129 of {\em London Mathematical Society Lecture Note Series}.
\newblock Cambridge University Press, Cambridge, 1990.

\bibitem{Pass}
D.S. Passman.
\newblock Groups with normal solvable hall $p'$-subgroups.
\newblock {\em Trans.Amer.Math.Soc.}, 123(1):99--111, 1966.

\bibitem{23Hall}
D.~O. Revin and E.~P. Vdovin.
\newblock On the number of classes of conjugate {H}all subgroups in finite
  simple groups.
\newblock {\em J. Algebra}, 324(12):3614--3652, 2010.

\bibitem{evenHall}
Danila~Olegovitch Revin and Evgenii~Petrovitch Vdovin.
\newblock Hall subgroups of finite groups.
\newblock In {\em Ischia group theory 2004}, volume 402 of {\em Contemp.
  Math.}, pages 229--263. Amer. Math. Soc., Providence, RI, 2006.

\bibitem{VdovinIntersSolv}
E.~P. Vdovin.
\newblock Regular orbits of solvable linear $p'$-groups.
\newblock {\em Siberian Electronic Mathematical Reports}, 4:345--360, 2007.

\bibitem{Vdo2}
E.~P. Vdovin.
\newblock On intersections of solvable hall subgroups in finite simple
  exceptional groups of lie type.
\newblock {\em Tr. Inst. Mat. Mekh.}, 19(3):62--70, 2013.

\bibitem{oddHall}
E.~P. Vdovin and D.~O. Revin.
\newblock Hall subgroups of odd order in finite groups.
\newblock {\em Algebra Logika}, 41(1):15--56, 118, 2002.

\bibitem{VdovinZenkov}
E.~P. Vdovin and V.~I. Zenkov.
\newblock On the intersection of solvable hall subgroups in finite groups.
\newblock {\em Proc. Stekl. Inst. Math. Suppl. 3}, pages 234--243, 2009.

\bibitem{zen}
V.~I. Zenkov.
\newblock Intersections of abelian subgroups in finite groups.
\newblock {\em Mat. Zametki}, 56(2):150--152, 1994.

\bibitem{Zen1}
V.~I. Zenkov.
\newblock Intersections of nilpotent subgroups in finite groups.
\newblock {\em Fund.Prikl.Mat.}, 2(2):1--92, 1996.

\bibitem{Zen2}
V.~I. Zenkov.
\newblock On the intersections of solvable hall subgroups in finite nonsolvable
  groups.
\newblock {\em Trudy IMM}, 13(2):86--89, 2007.

\end{thebibliography}

\textsc{
\phantom{G}
\medskip \\
Anton Baykalov\\
Department of Mathematics\\
The University of Auckland\\
Private Bag 92019\\
Auckland 1142\\
New Zealand
}

a.baykalov@auckland.ac.nz
\medskip \\
\textsc{Evgeny Vdovin\footnote{The second author is supported by RFBR grant No.\ 18-01-00752}\\
Sobolev Institute of Mathematics\\
and\\
Novosibirsk State University\\
Novosibirsk 630090\\
Russia
}

vdovin@math.nsc.ru
\medskip \\
\textsc{Victor Zenkov\\
N.N. Krasovskii Institute of Mathematics and Mechanics\\
16 S.Kovalevskaya Str.\\
Ekaterinburg  620108 \\
Russia}

zenkov@imm.uran.ru

%----------------------------------------------------------------------------------------

\end{document}